\documentclass[11pt]{elsarticle}
\makeatletter
\def\ps@pprintTitle{
	\let\@oddhead\@empty
	\let\@evenhead\@empty
	\def\@oddfoot{\footnotesize\itshape\hfill}
	\let\@evenfoot\@oddfoot
}
\makeatother
\usepackage{enumerate}
\usepackage{enumitem}
\usepackage{mathrsfs}
\usepackage{url}
\usepackage{comment}
\usepackage{graphicx}
\usepackage{float}
\usepackage{upgreek}
\usepackage{tikz}
\usepackage{grffile}
\usetikzlibrary{arrows,shapes,trees,backgrounds}
\usepackage{fancyhdr}
\usepackage{amsfonts, amsmath, amssymb, amscd, amsthm, bm, cancel}
\usepackage[
linktocpage=true,colorlinks,citecolor=magenta,linkcolor=blue,urlcolor=magenta]{hyperref}
\usepackage[margin=1.1in]{geometry}
\numberwithin{equation}{section}
\allowdisplaybreaks
\usepackage{setspace}
\newtheorem{prop}{Proposition}[section]
\newtheorem{thm}[prop]{Theorem}
\newtheorem{lem}[prop]{Lemma}
\newtheorem{cor}[prop]{Corollary}

\theoremstyle{definition}

\theoremstyle{remark}

\journal{journal}

\begin{document}

\begin{frontmatter}

\title{Global convergence for time-periodic  systems with negative feedback and applications }

\author[a]{Yi Wang\fnref{W}} 
\author[a]{Wenji Wu\fnref{W}} 
\affiliation[a]{organization={School of Mathematical Sciences, University of Science and Technology of China},
	addressline={wangyi@ustc.edu.cn,wuwenji@mail.ustc.edu.cn}, 
	city={Hefei},
	postcode={230026}, 
	state={Anhui},
	country={China}}
\author[b]{Hui Zhou\fnref{Z}} 

\affiliation[b]{organization={School of Mathematics and Statistics, Hefei Normal University},
addressline={zhouh16@mail.ustc.edu.cn}, 
	city={Hefei},
	postcode={230601}, 
	state={Anhui},
	country={China}}   

\fntext[W]{Supported by the National Natural Science Foundation of China (No.12331006), the Strategic Priority Research Program of CAS (No.XDB0900100) and the National Key R\&D Program of China	(No.2024YFA1013603, 2024YFA1013600).}
\fntext[Z]{Supported by Natural Science Foundation for Anhui Province Universities (2023AH040161).}
\cortext[W]{Corresponding author: zhouh16@mail.ustc.edu.cn (H. Zhou).}

\begin{abstract}
	For the discrete-time dynamical system 
	generated by the  Poincar\'e map $T$ of  a time-periodic closed-loop negative feedback system, we present an amenable condition which enables us to obtain the global convergence of the orbits. This yields  the  global convergence to the  harmonic periodic solutions of the corresponding time-periodic systems with negative feedback. Our approach is motivated by embedding the negative feedback system into a larger time-periodic monotone dynamical systems. We further utilize the theoretical results to obtain the global convergence to  periodic solutions for the time periodically-forced gene regulatory models. Numerical simulations  are  exhibited  to illustrate the feasibility of our  theoretical results for this model.
\end{abstract}


\end{frontmatter}

\section{Introduction}

Feedback control systems are crucial in both engineered and natural biological systems, ensuring stability and adaptability, or maintaining balance and optimal performance \cite{Angeli-Sontag-2003,cyclic-feedback-2020,monotone-cyclic-feedback-Mallet-1996,monotone-cyclic-feedback-Mallet-1990}. In particular, negative feedback systems are extensively utilized in biological processes  \cite{,de-Mottoni-Schiaffino-1981,Existence-periodic-solutions-negative-feedback-1977,Transversality-2017}. Meanwhile, periodic phenomena arise naturally in  population biology when day-night cycles or seasonal variation in parameters are accounted for \cite{de-Mottoni-Schiaffino-1981,Smith-1986,Zhao-book-2017}. In the present paper, we consider the following time-periodic cyclic system with negative feedback: 
\begin{equation}\label{cyclic-system}
	\left\lbrace
	\begin{aligned}
		\dot{x}_1&=F_1\left(t,x_1,x_n\right),& \\
		\dot{x}_i&=F_i\left(t,x_i,x_{i-1}\right),& \, 2 \leqslant i \leqslant n-1 ,\\
		\dot{x}_n&=F_n\left(t,x_n,x_{n-1}\right),&
	\end{aligned}
	\right.
\end{equation}
where the nonlinearity $F=\left(F_1,F_2,\cdots,F_n\right)$ is continuous, together with their partial derivatives with respect to $x$, for any $(t,x) \in  \mathbb{R}_+ \times X $. Here  $\mathbb{R}_+ =[0,\infty)$  and $X $ is a nonempty order-convex subset of $ \mathbb{R}^n$ (see Section \ref{section2}).  Moreover, there exists $\delta_i \in \{-1,+1\}$ such that 
\begin{equation*}\label{cyclic-system-character}
	\delta_i \frac{\partial F_i\left(t,x_i,x_{i-1}\right)}{\partial x_{i-1}} >0 ,\, \,\,\, \mathrm{for \, all }\,\, (t,x_i,x_{i-1}) \in \mathbb{R}_+ \times X \times X   \, \, \mathrm{and} \,\, 1 \leqslant i \leqslant n.
\end{equation*}
It is also assumed that there exists $\tau>0 $ such that 
\begin{equation}\label{period-cyclic-system}
	F(t,x)=F\left(t+\tau,x\right),\, \,(t,x)\in \mathbb{R}_+ \times X.
\end{equation}

For the autonomous case, that is, $F$ does not depend on $t$ in \eqref{cyclic-system},  Mallet-Paret and Smith \cite{monotone-cyclic-feedback-Mallet-1990} have established  the remarkable Poincar\'e-Bendixson theory, which yields that the $\omega$-limit set of any bounded orbit must be a nontrival periodic orbit if it contains no equilibrium. Following \cite{monotone-cyclic-feedback-Mallet-1990}, we call 
 system \eqref{cyclic-system}-\eqref{period-cyclic-system} 
a time-periodic \textit{monotone cyclic feedback system} (MCFS).  Let $\Delta= \delta_1 \delta_2 \cdots \delta_n =-1 $. Then system \eqref{cyclic-system}-\eqref{period-cyclic-system} turns out to be a time-periodic MCFS with  negative feedback. It is worth pointing out that such system is not monotone in the classical sense of Hirsch \cite{Stability-convergence-1988,Monotone-2004,Monotone-1995} with respect to certain usual convex cones. 

 Due to the complexity introduced by external periodic driving, very few studies have concentrated on dynamics of  the time-periodic MCFS with negative feedback.  To the best of our knowledge, Tere\v{s}\v{c}\'ak \cite{Terescak-1994} showed  that any $\omega$-limit set of the Poincar\'e map associated with system \eqref{cyclic-system}-\eqref{period-cyclic-system} can be embedded into  the plane (see also \cite{time-period-strongly-2-cooperative-system-2024} for generalized time-periodic  systems  with negative feedback). In the present paper, we will focus on the global  asymptotic behavior of the time-periodic negative feedback system \eqref{cyclic-system}-\eqref{period-cyclic-system} that can be viewed as a closed-loop system 
\begin{equation}\label{f(t, x, h(x))-1}
	\dot{x} = F(t,x) \equiv f(t,x,h(x)),
\end{equation}
which satisfies the following assumptions (see more details in Section \ref{section2}):
\begin{enumerate}
	\item[($\mathbf{A1}$)]   For each   $u \in U \subset \mathbb{R}^m$,  $\dot{x}=f(t, x,u)$ can generate a  time-periodic monotone system on $X$ in the usual order relation sense.
	\item[($\mathbf{A2}$)]  For each $(t,x) \in \mathbb{R}_+ \times  X $, $f(t,x, u)$ is increasing in $u $, relative to an order relation  \textquotedblleft  $ \leqslant_U $\textquotedblright.
	\item[($\mathbf{A3}$)]  The output  function $h:X \rightarrow U $
	is decreasing.   
\end{enumerate}
Here, the set $U \subset \mathbb{R}^m $ is called the input space and      \textquotedblleft  $ \leqslant_U $\textquotedblright   \, is the partial order in $U$,  and the assumption ($\mathbf{A3}$) exhibits characteristic of the  negative feedback.  

A prototypical example of system  \eqref{f(t, x, h(x))-1} is the following gene regulatory models:
\begin{equation}\label{gene-regulatory-model}
	\left\lbrace
	\begin{aligned}
		\dot{x_1}&=g(x_n)-\alpha_1(t)\, x_1 ,&\\
		\dot{x}_i&=x_{i-1}-\alpha_i(t) \,x_i,& \, 2 \leqslant i \leqslant n-1 ,\\
		\dot{x}_n&=x_{n-1} -\alpha_n(t)\,x_n,&
	\end{aligned}
\right.
\end{equation} 
where $\alpha_i(t)=\alpha_i(t+\tau) >0 $ is continuous for any $i =1, \cdots, n $, and $g:\mathbb{R}_+ \to \mathbb{R}_+  $ is continuously differentiable and satisfies $g(0) >0 $ and its derivative $g' < 0 $ (see \cite{cyclicgene-1987} for more details).

For the autonomous case,  Enciso et al.\,\cite{nonmonotone-systems-decomposable-2006} investigated the  asymptotic behavior of the negative feedback system \eqref{f(t, x, h(x))-1}. Under the   assumptions ($\mathbf{A1}$)-($\mathbf{A3}$)  with $f$ being independent of time $t$, they constructed some  positively invariant subset $B \subset X $; and moreover, under the additional  condition
\begin{equation*}\label{eql/4}
	a ,b \in B  \,\,\, \mathrm{with}	\,\,\, a \leqslant b  \,\, \,\mathrm{and}	\,\, \, f(a,h(b))=0=f(b,h(a)) \Longrightarrow a=b,
\end{equation*}
they have obtained  the  global convergence of any solution for such autonomous system.   

In this paper, we will investigate the dynamics of the periodic map (i.e.\;Poincar\'e map), devoted by $T$,  for  the time-periodic closed-loop negative feedback system \eqref{f(t, x, h(x))-1}. Motivated by the work  in \cite{nonmonotone-systems-decomposable-2006}, we will  embed the discrete-time dynamical system $\{T^n\}_{n \geqslant 0 }$ on $X$ into an extended  discrete-time  dynamical system $\{{\tilde{T}}^n\}_{n \geqslant 0 }$, which is defined  on the Cartesian Square $X \times X $, generated by the time-periodic symmetric system 
\begin{equation}\label{f(t,x,h(y))}
	\left\lbrace
	\begin{aligned}
		\dot{x}&=f\left(t,x,h(y)\right), \\
		\dot{y}&=f\left(t,y,h(x)\right),
	\end{aligned}
	\right.
\end{equation} such that $\tilde{T}(x,x)=(Tx,Tx)$ on the invariant diagonal of $X \times  X $. This  enables us to  find a positively invariant region $B$ for the  Poincar\'e map $T$ associated with system \eqref{f(t, x, h(x))-1}.  Furthermore, we provide an amenable condition that, for any 
two $\tau$-periodic continuous functions $a(\cdot), \, b(\cdot)$ with $a(0),b(0) \in B $ and $a(\cdot)\leqslant_K b(\cdot) $, there holds
\begin{eqnarray*}\label{amenable condition}
	\int_{0}^{\tau} f\left(t,a(t),h(b(t))\right) \mathrm{d} t=0=\int_{0}^{\tau} f\left(t,b(t),h(a(t))\right) \mathrm{d} t \Longrightarrow a(\cdot)=b(\cdot),
\end{eqnarray*}		
by which  we will show the global convergence of the orbits for the discrete-time dynamical system $\{T^n\}_{n \geqslant 0 }$ (see Theorem \ref{global-convergence}). 
In the terminology  of differential equations,  the global convergence to a unique $\tau$-periodic solution is thus obtained for  system \eqref{f(t, x, h(x))-1} (see Corollary \ref{global-convergence-system}). 

By applying our theoretical results to the time periodically-forced gene regulatory system  \eqref{gene-regulatory-model} with  $\alpha_{i}(t)= \alpha_{i},1 \leqslant i \leqslant n-1$, we obtain that any solution will approach to a unique $\tau$-periodic solution provided that the following condition holds: 
$$     \max\left\{-g'(u): 0 \leqslant u \leqslant  \alpha ^{-1} g(0) \right\} < \alpha ,  \,\,\,\,\,  \mathrm{where} \,\,\,\,  \alpha=  \prod_{i=1}^{n} \alpha_i   \,\,\,\,  \mathrm{and}  \,\,\,\,\, \alpha_{n}=\min_{0 \leqslant t \leqslant \tau }\alpha_{n}(t). \leqno(\mathbf{H})$$
Moreover,  numerical simulations  for the dynamics of the  time-periodic gene regulatory model are illustrated to show the obtained theoretical results for periodic system \eqref{f(t, x, h(x))-1}.

This paper is organized as follows. In 	Section \ref{section2}, we  introduce some notations and preliminaries, and state main results for the time-periodic closed-loop system
\eqref{f(t, x, h(x))-1} with negative feedback.  In Section \ref{section3},  we give the proof of our main results. Finally, we explore  the dynamics of  time periodically-forced  gene regulatory models  in Section  \ref{section4}.

\section{Notations and main results}\label{section2}

Let $\mathbb{R}^n$ be ordered by $\leqslant_K$ generated by a cone $K$ with nonempty interior. We write $ x \leqslant_K y $ if $y-x \in K ,  x <_K y $ if $y-x \in K $ and $y \neq x $, and $ x \ll_K y $ if $y-x \in \mathrm{Int} (K) $, where  $\mathrm{Int} (K) $ is the interior of $K$. Denote by $K^*$ the cone dual to $K$. If $ u,v \in \mathbb{R}^n $ satisfy $u <_K v $, then $[u,v]_K=\{x\in \mathbb{R}^n : u \leqslant_K x \leqslant_K v \}$ is called the order interval generated by $u$ and $v$. It is well known that order intervals in finite dimensional spaces are bounded. The set $X$ is said to be order-convex if the order interval $[u,v]_K  \subset X$ whenever $u,v \in X $ satisfy $u<_K v$. Let the nonempty state space  $X \subset \mathbb{R}^n $ be  order-convex and  $U \subset \mathbb{R}^m$ be the input space.  Denote by
 \textquotedblleft  $ \leqslant_U $\textquotedblright  \, the partial order in $U$. Hereafter, we write $C\left([0,\tau],X\right)$ as the space of all continuous functions on $[0,\tau]$ taking values in $X$. 

The time-periodic negative feedback system
\begin{eqnarray}\label{f(t,x,h(x))}
	\dot{x}  = F(t,x) \equiv f(t,x,h(x)) ,\,\, (t,x) \in  \mathbb{R}_+ \times X 
\end{eqnarray}
can be viewed as a closed-loop system 
\begin{equation*}\label{f(t,x,u)}
		\left\lbrace
	\begin{aligned}
		\dot{x}&=f(t,x,u),\,\, u \in U,& \\
		y&=h(x)& 
	\end{aligned}
		\right.
\end{equation*}
consisting of an open-loop  with input-output system 
by identifying input and output $y=u$ (see \cite{Angeli-Sontag-2003} for more details). 

Assume that $f:\mathbb{R}_+ \times X \times U \to \mathbb{R}^n $ and $h: X \to U $ are continuous and satisfy
\begin{enumerate}
	\item[($\mathbf{A1}$)]    For each  $(t,u) \in \mathbb{R}_+ \times U,\, x \mapsto f(t,x,u)$ is quasimonotone in the sense that:  for any $x,y \in X $, $\lambda \in K^* $ with $x \leqslant_K y $ and $ \lambda(x)=\lambda(y) $ implies $\lambda\left(f\left(t,x,u\right)\right) \leqslant \lambda\left(f\left(t,y,u\right)\right)$.
	\item[($\mathbf{A2}$)]   For each  $ (t,x) \in \mathbb{R}_+ \times X,\, u \leqslant_U v \Longrightarrow  f(t,x,u) \leqslant_K f(t,x,v) $.
	\item[($\mathbf{A3}$)]  For any $ x,y \in X$, $ x \leqslant_K y \Longrightarrow h(y) \leqslant_U h(x) $.
\end{enumerate}

Write $\psi(t,t_0,x_0)$ as the unique solution for system \eqref{f(t,x,h(x))} satisfying $\psi(t_0,t_0,x_0)=x_0$. We will assume without further mentioning that the domain of $\psi(t,t_0,x_0)$ includes $[t_0,\infty)$ in case $x_0 \in X $.
We embed the  closed-loop system (\ref{f(t,x,h(x))}) into a  larger symmetric system on the Cartesian Square $X \times X $ as
\begin{equation}\label{f(t,x,h(y)),f(t,y,h(x))}
	\left\lbrace
	\begin{aligned}
		\dot{x}&=f(t,x,h(y)), \\
		\dot{y}&=f(t,y,h(x)).
	\end{aligned}
	\right.
\end{equation} 

For $z_0 \in X \times X$, denote  $\phi(t,t_0,z_0)$ be  the unique solution of system \eqref{f(t,x,h(y)),f(t,y,h(x))} satisfying  $\phi(t_0,t_0,z_0)=z_0 $. Due to the uniqueness of solutions, the diagonal 

$$D=\{(x,x): x \in X\}$$
is positively invariant for (\ref{f(t,x,h(y)),f(t,y,h(x))}). More precisely, $ \phi(t,t_0,z_0) \in D $ whenever $ z_0=(x_0,x_0) \in D $, for all $ t_0 \geqslant 0 $. In particular,  $\phi(t,t_0,z_0)=\left(\psi (t,t_0,x_0),\psi(t,t_0,x_0)\right)$.

It is convenient to define the fundamental object of study in this paper, i.e., the Poincar\'e map for the $\tau$-periodic system
\eqref{f(t,x,h(x))}:
\begin{eqnarray*}\label{Poincare-map-T}
      T(x)=\psi(\tau,0,x),\,\, \mathrm{for\, any} \,\, x \in X.
\end{eqnarray*}
It is known that $T $ is a $C^1 $-diffeomorphism onto its image which is orientation preserving. According to the existence and uniqueness of the solution of the initial value problem and the periodicity of the vector field, one has 
\begin{eqnarray*}\label{T^n(z)}
	T^n(x)=\psi(n\tau,0,x),\,\, \mathrm{for\, any} \,\, n \geqslant 1,
\end{eqnarray*}
where $T^n=\underbrace{T \circ T \circ \cdots \circ T}_{n}$ is an $n$-order composite mapping.\\
Furthermore, we define  the Poincar\'e map $\tilde{T}$ for the $\tau$-periodic system \eqref{f(t,x,h(y)),f(t,y,h(x))} as
\begin{eqnarray*}\label{Poincare-map-tildeT}
	\tilde{T}(z)=\phi(\tau,0,z),\,\, \mathrm{for\, any} \,\,  z \in X \times X .
\end{eqnarray*}
In particular, if $z=(x,x) \in D $, then $\tilde{T}(x,x)=(Tx,Tx)$.

For the Poincar\'e map ${T}$, the orbit of $x_0$ is $O_{T}(x_0)=\{T^nx_0: n \geqslant 0 \}$, and the $\omega$-limit set of $x_0$ is $\omega_{T}(x_0)=\{y \in X : \mathrm{there \, exist}\; n_k \to \infty \; \mathrm{such \, that}\; T^{n_k}x_0 \to y ,\, \mathrm{as}\ k \to \infty \}.$ A subset $A \subset X $ is called invariant with respect to $T$ (resp. positively invariant) if $TA=A$\;(resp. $TA \subset A$). Obviously, if $O_T(x_0)$ has compact closure, then $\omega_T(x_0)$ is nonempty, compact and invariant.
Call  $x_0 $ a fixed point of $T$ if $Tx_0=x_0 $. Similarly, for the Poincar\'e map $\tilde{T}$,  we write $O_{\tilde{T}}(z_0),\, \omega_{\tilde{T}}(z_0) $ for the orbit and the $\omega$-limit set of $z_0 \in X \times X$, respectively.
 
 By virtue of  ($\mathbf{A1}$)-($\mathbf{A3}$),  the Poincar\'e map $\tilde{T}$ for the extended system (\ref{f(t,x,h(y)),f(t,y,h(x))}) can generate a discrete-time  monotone system on $X \times X \subset  \mathbb{R}^n \times \mathbb{R}^n $ with regard to the cone $C=K \times (-K) $ as below, where $C$ gives rise to the order relation 
 $$(x,y) \leqslant_C (\bar{x},\bar{y}) \iff x \leqslant_K \bar{x} \,\,\,\, \mathrm{and} \,\,\,\,  \bar{y} \leqslant_K y .$$
 The dual cone $C^*$ can be represented as $K^* \times (-K)^*$,  where $(\lambda,-\mu)(x,y)=\lambda(x)-\mu(y) $ holds for all $x,y \in \mathbb{R}^n$ and $\left(\lambda,-\mu \right)\in K^* \times (-K)^*$ (see  Lemma \ref{monotone-system}).

Now we introduce the following amenable condition  and  state our main results as follows.
\begin{enumerate}
	\item[($\mathbf{A4}$)]   There exist  $x_0, y_0  \in X$   such that $x_0\leqslant_K y_0 $ and 
		\begin{eqnarray}\label{integral-condition1}
		\int_{0}^{\tau} f\left(t,y_0(t),h(x_0(t))\right) \mathrm{d}t \leqslant_K 0 \leqslant_K \int_{0}^{\tau} f\left(t,x_0(t),h(y_0(t))\right) \mathrm{d}t ,
	\end{eqnarray}	
	where $\left(x_0(t),y_0(t)\right)=\phi(t,0,(x_0,y_0))$ is the solution of system
		\eqref{f(t,x,h(y)),f(t,y,h(x))}. Moreover,
for any 
	two $\tau$-periodic continuous functions $a(\cdot), \, b(\cdot) \in C\left([0,\tau],X\right) $ with $a(0),b(0) \in [x_0,y_0]_K $ and $a(\cdot)\leqslant_K b(\cdot) $, there holds
	\begin{eqnarray}\label{integral-condition2}
		\int_{0}^{\tau} f\left(t,a(t),h(b(t))\right) \mathrm{d} t=0=\int_{0}^{\tau} f\left(t,b(t),h(a(t))\right) \mathrm{d} t \Longrightarrow a(\cdot)=b(\cdot).
	\end{eqnarray}		
	
\end{enumerate}

\begin{thm}\label{global-convergence}
	Assume  $\mathrm{(\mathbf{A1}) }$-$\mathrm{(\mathbf{A4})}$  hold. Then  the Poincar\'e map $T$ for system \eqref{f(t,x,h(x))} possesses a unique fixed point  $r \in [x_0,y_0]_K$ such that
	$$\omega_{T}\left([x_0,y_0]_K\right)=\{r\}.$$
\end{thm} 

In the terminology of  differential equations, we have the following result for  system \eqref{f(t,x,h(x))}.
\begin{cor}\label{global-convergence-system}
	Assume  $\mathrm{(\mathbf{A1}) }$-$\mathrm{(\mathbf{A4})}$  hold. Then  system \eqref{f(t,x,h(x))} has a unique $\tau$-periodic solution $r(t)$.
	Moreover, for any $\bar{x} \in [x_0,y_0]_K$,
	$$\lVert \psi(t,0,\bar{x})-r(t)\lVert\ \rightarrow 0 , \, \,\, \, \, \mathrm{as} \,\, t \to \infty .$$
\end{cor} 

\section{Proof of the main results}\label{section3}
Throughout this section, we always assume that $\mathrm{(\mathbf{A1})}$-$\mathrm{(\mathbf{A4}) }$ hold. In order to prove our main results, we need the following Lemmas.

\begin{lem}\label{monotone-system}
The Poincar\'e map $\tilde{T}$  for \eqref{f(t,x,h(y)),f(t,y,h(x))} generates a discrete-time monotone system on $X \times X $ with respect to $\leqslant_C $.		
\end{lem}

\begin{proof}
	Let $G(t,(x,y))=\left(f(t,x,h(y)),f(t,y,h(x))\right)$. To prove the monotonicity of the Poincar\'e map $\tilde{T}$, one  only needs to show that $G(t,(x,y))$ satisfies  the quasimonotone condition relative to the cone $C$ (see \cite[Section 3, Theorem 3.2]{Monotone-2004}). To this purpose, for any $(x,y) \leqslant_C (\bar{x},\bar{y}) $ and $(\lambda,-\mu) \in C^* $ with $(\lambda,-\mu)(x,y)=(\lambda,-\mu)(\bar{x},\bar{y})$, we must verify that 
	\begin{eqnarray*}
		(\lambda,-\mu)G\left(t,(x,y)\right) \leqslant (\lambda,-\mu)G\left(t,(\bar{x},\bar{y})\right), \, \, \mathrm{for \,\,  any} \,\, t \geqslant0.
	\end{eqnarray*}	
	In fact,  $(\lambda,-\mu)(x,y)=(\lambda,-\mu)(\bar{x},\bar{y})$ and $\lambda,\mu \in K^* $ imply that $$0 \leqslant \lambda(\bar{x}-x)=\mu(\bar{y}-y) \leqslant 0 , $$
and hence, $$\lambda(x)=\lambda(\bar{x}) \,\,\, \mathrm{and}\,\,\, \mu(y)=\mu(\bar{y}) .$$ 
Recall  that $x \leqslant_K \bar{x} $ and  $ \bar{y} \leqslant_K y $. Then, one has  
	\begin{align*}
		(\lambda,-\mu)G\left(t,(x,y)\right)&=\lambda\left(f(t,x,h(y))\right)-\mu\left(f(t,y,h(x))\right) \\
		&\leqslant \lambda\left(f(t,x,h(\bar{y}))\right)-\mu\left(f(t,y,h(\bar{x}))\right) \\
		&\leqslant \lambda\left(f(t,\bar{x},h(\bar{y}))\right)-\mu\left(f(t,\bar{y},h(\bar{x}))\right) \\
		&= (\lambda,-\mu)G\left(t,(\bar{x},\bar{y})\right),
	\end{align*}		
	where the first inequality  is due to the monotonicity of $f$ and $h$ with respect to ($\mathbf{A2}$)-($\mathbf{A3}$),	and the second inequality  follows from ($\mathbf{A1}$).
\end{proof}

\begin{lem}\label{positively-invariant}
	If \eqref{integral-condition1} holds, then the order interval $$I=[(x_0,y_0),(y_0,x_0)]_C=\{(x,y):(x_0,y_0) \leqslant_C (x,y) \leqslant_C (y_0,x_0)\}$$
is positively invariant with respect to   $\tilde{T}$ for system	\eqref{f(t,x,h(y)),f(t,y,h(x))}.
 In particular, 
$ [x_0,y_0]_K $ is positively invariant with respect to   $T$ for system  \eqref{f(t,x,h(x))}.

\end{lem} 

\begin{proof} 
	Note that $x_0 \leqslant_K y_0 $ implies $(x_0,y_0) \leqslant_C (y_0,x_0)$. Denote $a_0=(x_0,y_0)$, $b_0=(y_0,x_0)$.
	By  \eqref{integral-condition1},  we obtain 
\begin{align*}
	\tilde{T}a_0-a_0&= \phi(\tau,0,a_0)-a_0\\
	&=	\int_{0}^{\tau} G\left(t,\phi(t,0,a_0)\right) \mathrm{d}t \\
	&=\left(\int_{0}^{\tau} f\left(t,x_0(t),h(y_0(t))\right) \mathrm{d}t, \int_{0}^{\tau} f\left(t,y_0(t),h(x_0(t))\right) \mathrm{d}t\right)  \geqslant_C \left(0,0\right),
\end{align*}		
	that is, $	\tilde{T}a_0 \geqslant_C a_0$.
	Observe also that system  \eqref{f(t,x,h(y)),f(t,y,h(x))}
	is symmetric. Then it yields that 
	$$\phi(t,0,b_0)=\left(y_0(t),x_0(t)\right), \,\,\,\mathrm{for\,\, any} \,\,\, t \geqslant 0 .$$
	Thus, by repeating the same arguments above, one can obtain 	$\tilde{T}b_0\leqslant_C b_0$. Consequently, for any $z \in I $, it follows from Lemma \ref{monotone-system} that  
	$$a_0 \leqslant_C \tilde{T}a_0 \leqslant_C \tilde{T}z \leqslant_C \tilde{T}b_0 \leqslant_C b_0 ,$$
	which entails that 
	$I $ is positively invariant for $\tilde{T}$.
	Note also that the diagonal 
	$$D=\{(x,x): x \in X\}$$
	is positively invariant for system (\ref{f(t,x,h(y)),f(t,y,h(x))}). 
	Then 
	$I \cap D $ is positively invariant for  $\tilde{T}$.  Since  $\{{\tilde{T}}^n\}_{n \geqslant 0 }$ reduces to  $\{{T}^n\}_{n \geqslant 0 }$ on $D$, we obtain that 
	$[x_0,y_0]_K $ is  positively invariant with respect to $T$.
\end{proof}

\begin{lem}\label{two-equilibria}
The Poincar\'e map $\tilde{T}$ has two fixed points $p,q \in I $ with $p \leqslant_C q $.
\end{lem} 

\begin{proof}
Given any $z\in I $, one has $a_0 \leqslant_C z \leqslant_C b_0 $. By virtue of  Lemmas \ref{monotone-system}-\ref{positively-invariant}, we  obtain 
	\begin{eqnarray*}\label{tilde{T}-bar{z}}
		a_0 \leqslant_C \tilde{T}a_0 \leqslant_C \tilde{T}z\leqslant_C \tilde{T}b_0 \leqslant_C b_0 .
	\end{eqnarray*}
	Observe that $O_{ \tilde{T}}(a_0), \,O_ {\tilde{T}}(b_0)\subset I $ have compact closure, then there exist $p,q \in I $ such that $\omega_{\tilde{T}}(a_0)=\{p\} $ and $\omega_{\tilde{T}}(b_0)=\{q\} $ by using monotone convergence  criterion (see \cite[Section 5, Lemma 5.3]{Monotone-2004}). Thus,
	\begin{eqnarray*}\label{tilde{T}^n-a_0}
		\tilde{T}^n a_0 \to p ,\; \tilde{T}^n b_0 \to q , \,\,\,\,  \mathrm{as} \,\, n \to \infty .
	\end{eqnarray*}
	Clearly, $p$ and $q $ are two fixed points of $\tilde{T}$ and $p \leqslant_C q $.
\end{proof}

\begin{figure}[H]
	\centering
	\includegraphics[scale=0.30]{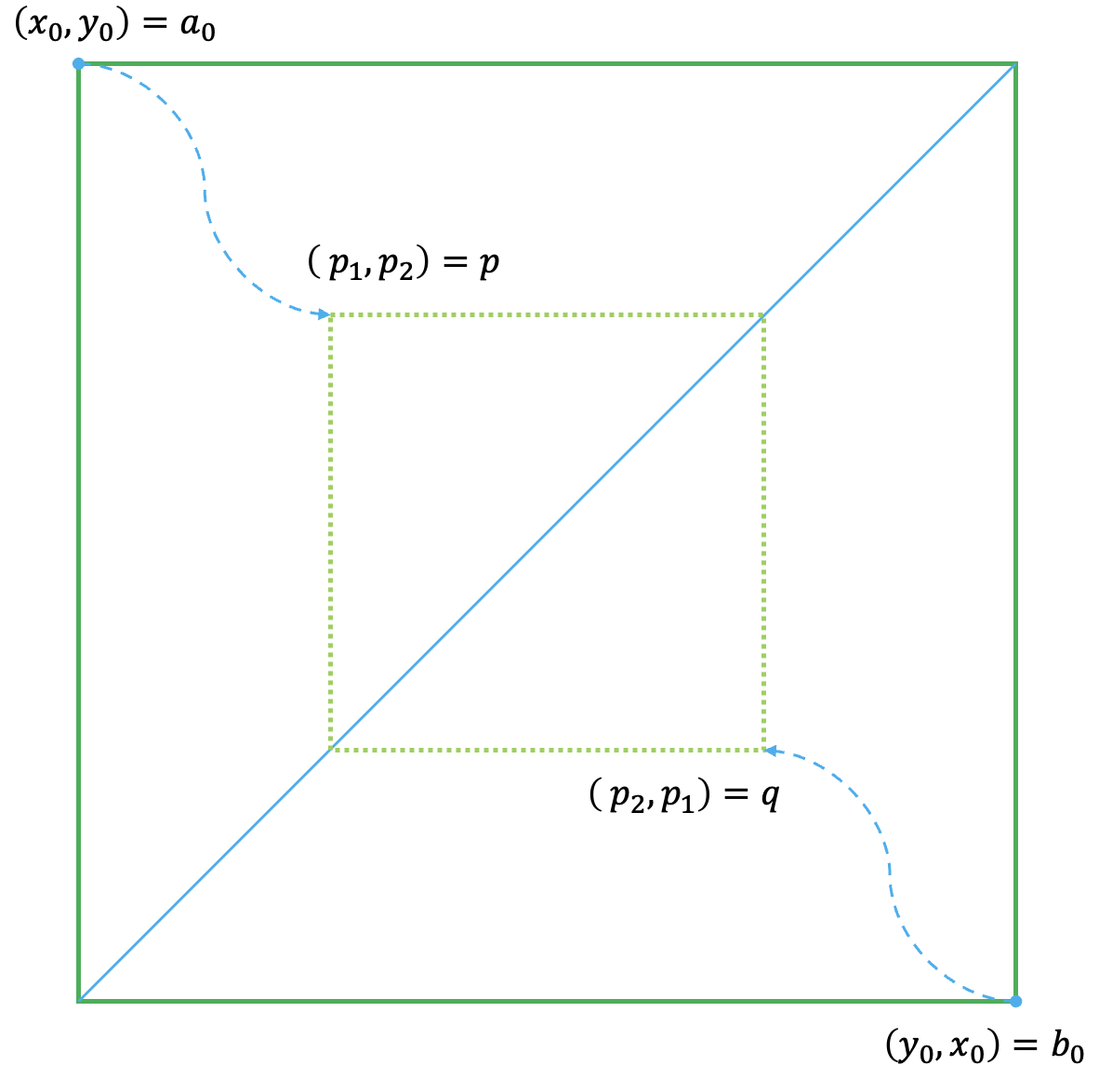}
	\caption{Orbit convergence: $\tilde{T}^n a_0 \uparrow p $ and $\tilde{T}^n b_0 \downarrow q$ with $\omega_{\tilde{T}} \left([a_0,b_0]_C\right)$ remained in the  order interval $[a_0,b_0]_C$.} 
\end{figure}

An immediate consequence of Lemma \ref{two-equilibria} is the following result.
\begin{lem}\label{two-tau-periodic-solutions}
System \eqref{f(t,x,h(y)),f(t,y,h(x))} possesses two $\tau$-periodic solutions $p(t)=(p_1(t),p_2(t))$, $q(t)=(p_2(t),p_1(t))$ satisfying $p(0)=p $, $q(0)=q $ and $p(t) \leqslant_C q(t) $, for any $ t \geqslant 0 $. 
Moreover,
$$\lVert \phi(t,0,a_0)-p(t)\lVert\ \rightarrow 0  \,\,\, \mathrm{and} \,\,\, \, \lVert \phi(t,0,b_0)-q(t)\lVert\ \rightarrow 0 , \, \,\,\, \mathrm{as} \,\, t \to \infty .$$
\end{lem} 

\begin{proof}
By Lemma \ref{two-equilibria}, 
$p$ and $q $ are two fixed points of $\tilde{T}$ satisfying  $p \leqslant_C q $.
Then 
$$p(t)=\phi(t,0,p)\,\,\, \mathrm{and} \,\,\, q(t)=\phi(t,0,q)$$ 
are two $\tau$-periodic solutions of the system \eqref{f(t,x,h(y)),f(t,y,h(x))}. 
	We will show 
$$\lVert \phi(t,0,a_0)-p(t)\lVert\ \rightarrow 0 , \, \, \, \mathrm{as}\,\,  t \to \infty .$$ The proof for $q(t)$ is analogous. 
	Since $\tilde{T}^n a_0 \to p $,  one has 
\begin{equation}\label{phi(t,t_0,tilde{T}^n a_0)-p(t)}
	\begin{aligned}
		\| \phi(t,0,\tilde{T}^n a_0)-p(t)\| &=
		\lVert \phi\left(t,0,\phi(n\tau,0,a_0) \right)-p(t)\lVert\  \\
		&=	\lVert \phi\left(t+n\tau,n\tau,\phi(n\tau,0,a_0)\right)-p(t)\lVert\  \\
		&=
		\lVert \phi(t+n\tau,0,a_0)-p(t)\lVert\ \to 0, 
	\end{aligned}
\end{equation}
	as $ n \to \infty $ uniformly  for $t \in [0,\tau]$.
	For any sequence $t_n \to \infty $, we write $t_n =m_n \tau+s_n $, where  $ m_n $ are nonnegative integers and $ s_n \in [0,\tau)$.
Then
\begin{equation*}
	\begin{aligned}
		\| \phi(t_n,0, a_0)-\phi(t_n,0, p)\| &=
		\lVert \phi(m_n\tau+s_n,0,a_0)-\phi(m_n\tau+s_n,0, p)\lVert\  \\
		&=	\lVert \phi(m_n\tau+s_n,0,a_0)-\phi(s_n,0, p)\lVert.
	\end{aligned}
\end{equation*}
	Together with  \eqref{phi(t,t_0,tilde{T}^n a_0)-p(t)}, we obtain that 
\begin{eqnarray*}
	\| \phi(t,0,a_0)-p(t)\| =\| \phi(t,0,a_0)-\phi(t,0,p)\| \rightarrow 0 , 
\end{eqnarray*}
as $ n \to \infty $.
Similarly, 
\begin{eqnarray*}
	\lVert \phi(t,0,b_0)-q(t)\lVert\  \rightarrow 0 , \, \,\,\, \mathrm{as} \,\,  t \to \infty .
\end{eqnarray*}
	Finally, donote $$p(t)=(p_1(t),p_2(t)) .$$
	 Then the symmetry of system \eqref{f(t,x,h(y)),f(t,y,h(x))} implies that 
	$$q(t)=(p_2(t),p_1(t)).$$
	Again, by Lemma \ref{monotone-system}, one has
\begin{eqnarray*} 
	p(t)=\phi(t,0,p) \leqslant_C \phi(t,0,q)=q(t),\, \,\, \mathrm{for\, \, any} \,\,\, t \geqslant 0 .
\end{eqnarray*}
	Thus, we have completed the proof.
\end{proof}	 

Now, we are ready to  prove Theorem \ref{global-convergence}.

\begin{proof}[Proof of Theorem \ref{global-convergence}]
	By Lemma \ref{two-tau-periodic-solutions}, $p(t)=(p_1(t),p_2(t))$, $q(t)=(p_2(t),p_1(t)) $ are two $\tau$-periodic solutions of system \eqref{f(t,x,h(y)),f(t,y,h(x))} with  $p_1(t) \leqslant_K p_2(t) ,\,\, \mathrm{for \,\, any }\,\, t \geqslant 0 $. Then $p_1(t)$, $p_2(t)$ are two $\tau$-periodic continuous  functions and satisfy 
\begin{equation*}\label{f(t,p_1(t),h-circ-p_2(t))}
		\left\lbrace
		\begin{aligned}
			\dot{p}_1(t)&=f\left(t,p_1(t),h (p_2(t) )\right), \\			
			\dot{p}_2(t)&=f\left(t,p_2(t),h (p_1(t)) \right).
		\end{aligned}
		\right.
\end{equation*} 
	Therefore, we obtain
	$$0=p_1(\tau)-p_1(0)=\int_{0}^{\tau} f\left(t,p_1(t),h (p_2(t)) \right) \mathrm{d}t ,$$
	and
	$$0=p_2(\tau)-p_2(0)=\int_{0}^{\tau} f\left(t,p_2(t),h (p_1(t) )\right) \mathrm{d}t .$$
	By virtue of  \eqref{integral-condition2} in  $\mathrm{(\mathbf{A4}) }$, there holds
	$$p_1(t)=p_2(t),\,\, \mathrm{for \,\, any} \,\,  t \geqslant 0 .$$  
	Hence,
	$$p=(p_1,p_2)=(p_2,p_1)=q  \overset{\triangle }{=}(r,r).$$ 
	Due to  Lemma \ref{two-equilibria}, it is clear  that  $p (=q)$ is the unique fixed point of  $\tilde{T}$  in $I$.
	Therefore, 	we can obtain  
    $$\omega_{\tilde{T}}\left(I\right)=\omega_{\tilde{T}}\left([a_0,b_0]_C\right)=\{p\} \left(=\{q\}\right).$$
   This implies that $r $ is the unique fixed point of $T$ in $ [x_0,y_0]_K $ and
   	$$ \omega_{T}\left([x_0,y_0]_K\right)=\{r\}.$$
Thus, we have proved Theorem \ref{global-convergence}.
\end{proof}

\section{Application to time-periodic gene regulatory models}$\ $
\label{section4}
In this section, we explore the dynamics of  time periodically-forced  gene regulatory models. Concretely, we prove the existence of periodic solutions and global convergence. In addition,  we illustrate the consistence of theoretical and actual results by numerical simulation.

A prototypical example of  Section \ref{section2}, treated in \cite{cyclicgene-1987}, is the gene regulatory system modeled by the equations
\begin{equation}\label{g(xn)-alpha1x1}
	\left\lbrace
	\begin{aligned}
		\dot{x}_1&=g(x_n)-\alpha_1 x_1 ,&\\
		\dot{x}_i&=x_{i-1}-\alpha_i x_i,& \, 2 \leqslant i \leqslant n-1 ,\\
		\dot{x}_n&=x_{n-1} -\alpha_n(t)x_n,&
	\end{aligned}
	\right.
\end{equation} 
on $ \mathbb{R}_+ ^n$, where $\alpha_i >0 $ for any $i = 1, \cdots , n-1 $,  and the $\tau$-periodic  function $\alpha_n(t)=\alpha_n(t+\tau) >0 $ is continuous, and $g:\mathbb{R}_+ \to \mathbb{R}_+  $ is  continuously differentiable and satisfies $g(0) >0 $ and $g' < 0 $.

Let $x=\left(x_1,x_2,\cdots,x_n\right)^T $, $h(x)\overset{\triangle }{=}g(x_n) $ and $ \tilde{h}(x)\overset{\triangle }{=}\left(g(x_n),0,\cdots,0\right)^T $, here $T$ means  the transpose. Then system \eqref{g(xn)-alpha1x1} could be rewritten as 
\begin{eqnarray*}\label{A(t)x+tilde{h}(x)}
	\dot{x}=A(t)\,x +\tilde{h}(x) = f(t,x,h(x)),
\end{eqnarray*}
where $A(t)=(a_{ij}(t))_{n \times n }$ is a quasipositive matrix  (i.e. $a_{ij}(t) \geqslant 0 $, $i \neq j $, $t \geqslant 0 $). 
With these hypotheses, the time-dependent  vector field  $f$ naturally satisfies assumptions 
($\mathbf{A1}$)-($\mathbf{A3}$). Take the cone $ K=\mathbb{R}_+ ^n $.
Let
\begin{eqnarray}\label{box-X=[x_0,y_0]}
 X\overset{\triangle }{=}\left[0,g(0)\left(\alpha_1^{-1},\alpha_1^{-1}\alpha_2^{-1},\cdots,\alpha_1^{-1}\alpha_2^{-1} \cdots \alpha_n^{-1}\right)^T\right]_K 
\end{eqnarray}
and introduce the assumption
$$     \max\left\{-g'(u): 0 \leqslant u \leqslant  \alpha ^{-1} g(0) \right\} < \alpha ,  \,\,\,\,\,  \mathrm{where} \,\,\,\,  \alpha=  \prod_{i=1}^{n} \alpha_i   \,\,\,\,  \mathrm{and}  \,\,\,\,\, \alpha_{n}=\min_{0 \leqslant t \leqslant \tau }\alpha_{n}(t). \leqno(\mathbf{H})$$
By applying Corollary \ref{global-convergence-system} to the gene regulatory model \eqref{g(xn)-alpha1x1}, we  have the following 
\begin{thm}\label{gene-regulatory} Assume $\mathbf{(H)}$
	holds.  Then system \eqref{g(xn)-alpha1x1} has a unique $\tau$-periodic solution $r(t) $ in X  such that for any $\bar{x} \in X $,
	$$\lVert \psi(t,0,\bar{x})-r(t)\lVert\ \rightarrow 0 , \, \,\,\, \mathrm{as} \,\, t \to \infty  , $$
	where $\psi(t,0,\bar{x})$ is the solution of system \eqref{g(xn)-alpha1x1} satisfying $\psi(0,0,\bar{x})=\bar{x}$. 
\end{thm}

\begin{proof}
We first  show that the solution $\psi(t,0,\bar{x})$ remains  in $X $ (see \eqref{box-X=[x_0,y_0]}), for any $\bar{x} \in X $ and $t \geqslant 0 $. Write $$\psi(t,0,\bar{x})=\left(x_1(t),x_2(t),\cdots,x_n(t)\right)\,\, \, \mathrm{and} \,\,\, \, \bar{x}=\left(\bar{x}_1,\bar{x}_2,\cdots,\bar{x}_n \right).$$  
	It is easy to see that $\psi(t,0,\bar{x}) \in \mathbb{R}_+ ^n $, whenever  $\bar{x} \in \mathbb{R}_+ ^n $. 
	Since  $g(x_n(t))-\alpha_1 x_1  \leqslant g(0)-\alpha_1 x_1 $, the standard comparison theorem (see \cite[Chapter 1, Theorem 1.4.1]{Differential-and-Integral-Inequalities-1969}) implies that  
	$$x_1(t) \leqslant e^{-\alpha_1t}\bar{x}_1+g(0)\alpha_1^{-1}\left(1-e^{-\alpha_1t}\right) \leqslant g(0)\alpha_1^{-1}.$$
	Similarly, one has 
	$$x_{i}(t) \leqslant  g(0)\alpha_1^{-1}\alpha_2^{-1}\cdots \alpha_{i}^{-1},\,\,\,\, \mathrm{for } \,\,\,  2 \leqslant i \leqslant n-1. $$
	Noticing that $x_{n-1}(t) -\alpha_n(t)x_n \leqslant g(0)\alpha_1^{-1}\alpha_2^{-1}\cdots \alpha_{n-1}^{-1} -\alpha_n x_n $, we  obtain 
	$$ x_{n}(t) 
	\leqslant g(0)\alpha_1^{-1}\alpha_2^{-1}\cdots \alpha_{n}^{-1}.$$
	Thus, $\psi(t,0,\bar{x})$ remains  in $X $  for any  $t \geqslant 0 $. 
	
	Next, we verify that system  \eqref{g(xn)-alpha1x1} satisfies  \eqref{integral-condition1} in  ($\mathbf{A4}$). Let
	$$x_0=0   , \, \,  y_0= g(0)\left(\alpha_1^{-1},\alpha_1^{-1}\alpha_2^{-1},\cdots,\alpha_1^{-1}\alpha_2^{-1} \cdots \alpha_n^{-1}\right)^T  \,\, \mathrm{in} \,\,\, X  \,\,\, \mathrm{and} \,\,\,  U=[0,g(0)] \,\,\,\,  \mathrm{in} \,\,\, \,  \mathbb{R}.$$
	Then, system \eqref{g(xn)-alpha1x1} can be imbedded in the extended symmetric system

		\begin{equation}\label{g(xn)-alpha1x1-imbedded}
			\begin{array}{cc}
			\left\lbrace
			\begin{aligned}
				\dot{x}_1&=g(y_n)-\alpha_1 x_1 ,\\
				\dot{x}_i&=x_{i-1}-\alpha_i x_i,\\
				\dot{x}_n&=x_{n-1} -\alpha_n(t)x_n,\\
				\dot{y}_1&=g(x_n)-\alpha_1 y_1 ,\\
				\dot{y}_i&=y_{i-1}-\alpha_i y_i, \\
				\dot{y}_n&=y_{n-1} -\alpha_n(t)y_n.
			\end{aligned}
			\right.
			& 2\leqslant i \leqslant n-1,
			\end{array}
		\end{equation}
		Let
		$ \left(x_0(t),y_0(t)\right) =\phi(t,0,(x_0,y_0))$
		denote the solution of system \eqref{g(xn)-alpha1x1-imbedded}. Then, one has
		$$	\int_{0}^{\tau} f\left(t,y_0(t),h(x_0(t))\right) \mathrm{d}t
		=y_0(\tau) -y_0(0) = y_0(\tau) -y_0 \leqslant 0  ,$$
		and 
		$$ \int_{0}^{\tau} f\left(t,x_0(t),h(y_0(t))\right) \mathrm{d}t 
		=x_0(\tau) -x_0(0) = x_0(\tau) -x_0 \geqslant 0 .$$
		This verifies \eqref{integral-condition1} in  ($\mathbf{A4}$).
		
		Finally, we prove that system  \eqref{g(xn)-alpha1x1}   satisfies   \eqref{integral-condition2} in ($\mathbf{A4}$). Let
		$a(\cdot),b(\cdot)  $ be two $\tau$-periodic continuous functions taking values in $X$ with $a(\cdot)\leqslant_K b(\cdot) $ and 
		\begin{eqnarray}\label{average-integral-condition}
			\int_{0}^{\tau} f\left(t,a(t),h(b(t))\right) \mathrm{d} t=0=\int_{0}^{\tau} f\left(t,b(t),h(a(t)\right)) \mathrm{d} t.
		\end{eqnarray}
		We will show  that 
		$a(t)=b(t)$ for any $ t \in [0,\tau] $. 
		To this end, we write 
		$$ a(t)=\left(a_1(t),a_2(t),\cdots,a_n(t)\right)^T\,\,\,\,  \mathrm{and} \,\,\, \,  b(t)=\left(b_1(t),b_2(t),\cdots,b_n(t)\right)^T.$$
		Then \eqref{average-integral-condition} implies that 
		\begin{equation*}\label{integral-equality}
			\begin{aligned}
				\displaystyle\int_{0}^{\tau} g(b_n(t))\, \mathrm{d}t&= \alpha_1 \cdot \displaystyle\int_{0}^{\tau} a_1(t)\,\mathrm{d}t ,
				\vspace{0.22cm} \\
				\displaystyle\int_{0}^{\tau} a_i(t)\,\mathrm{d}t&= \alpha_{i+1}  \cdot \displaystyle\int_{0}^{\tau} a_{i+1}(t)\,\mathrm{d}t ,\,\, \,\,1 \leqslant i \leqslant n-2,\\
				\displaystyle\int_{0}^{\tau} a_{n-1}(t)\,\mathrm{d}t&= 	\displaystyle\int_{0}^{\tau} \alpha_{n}(t)a_n(t)\,\mathrm{d}t .
			\end{aligned}
		\end{equation*}
			Hence,  by iterating these equalities, one has 
				$$	\int_{0}^{\tau} g(b_n(t))\, \mathrm{d} t= \prod_{i=1}^{n-1} \alpha_i  \cdot \int_{0}^{\tau} \alpha_{n}(t)a_n(t)\,\mathrm{d}t .$$
				Similarly, one can also obtain  symmetrically
				$$	\int_{0}^{\tau} g(a_n(t))\, \mathrm{d} t=\prod_{i=1}^{n-1} \alpha_i  \cdot \int_{0}^{\tau} \alpha_{n}(t)b_n(t)\,\mathrm{d}t .$$
				Consequently, 
				\begin{equation}\label{equality}
					\begin{aligned}
						\int_{0}^{\tau}  \Big[g(a_n(t))+    \prod_{i=1}^{n-1} \alpha_i  \cdot \alpha_{n}(t)a_n(t)  \Big]\,\mathrm{d}t 
						=	\int_{0}^{\tau}  \Big[g(b_n(t))+ \prod_{i=1}^{n-1} \alpha_i   \cdot \alpha_{n}(t)b_n(t)   \Big]\,\mathrm{d}t .
					\end{aligned}
				\end{equation}
			Now, define the following functional on $C\left([0,\tau],\mathbb{R} \right) $ as 
				$$J: C\left([0,\tau],\mathbb{R} \right) \to \mathbb{R},\, \xi \mapsto 	\int_{0}^{\tau}  \Big[g(\xi (t))+   \prod_{i=1}^{n-1} \alpha_i  \cdot  \alpha_{n}(t)\xi(t) \Big]\,\mathrm{d}t  .$$
		Then, the  Frech\'et derivative of $J$ at $\xi $ is 	
				\begin{equation}\label{Frechet-derivative}
					\begin{aligned}
						\langle DJ(\xi),\varphi \rangle
						=&
						\int_{0}^{\tau} \Big[g'(\xi (t))+ \prod_{i=1}^{n-1} \alpha_i   \cdot \alpha_{n}(t)\Big]\varphi(t)\,\mathrm{d}t .
					\end{aligned}
				\end{equation}			
	Together \eqref{equality}, there holds 
					\begin{align}
						0 \overset{\eqref{equality}}{=\joinrel=}&J(b_n)-J(a_n)  \notag\\
						=\joinrel=&\int_{0}^{1}\langle DJ\left((1-s)a_n+sb_n\right),b_n-a_n \rangle \mathrm{d}s
						\notag \\
						\overset{\eqref{Frechet-derivative}}{=\joinrel=}	&\int_{0}^{1} \int_{0}^{\tau}  \Big[\,g'\left((1-s)a_n(t)+sb_n(t)\right)+ \prod_{i=1}^{n-1} \alpha_i  \cdot \alpha_{n}(t)\, \Big]\left(b_n(t)-a_n(t)\right)\mathrm{d}t \,\mathrm{d}s. 
						\label{contradiction}
					\end{align}
				Suppose $a_n(\cdot) \neq b_n(\cdot)$. Then one can find  an interval $ [c,d] \subset [0,\tau]$ such that 
				$$ a_n(t) < b_n(t) ,\, \,\, \mathrm{for\,\,any} \,\,\, t \in [c,d].$$
				Then the  assumption $\mathbf{(H)}$ entails that 
				\begin{equation*}			
					\int_{c}^d\,  \Big[ \,g'\left((1-s)a_n(t)+sb_n(t)\right)+ \prod_{i=1}^{n-1} \alpha_i \cdot \alpha_{n}(t)\, \Big]\left(b_n(t)-a_n(t)\right)\mathrm{d}t  
					>  0,
				\end{equation*}					
			contradicting \eqref{contradiction}, which completes the proof.
			\end{proof}
			
 Finally, we present some numerical simulations   to illustrate our  main results in system \eqref{g(xn)-alpha1x1}. We take $n=3$,   $\alpha_1=2,\; \alpha_2=1 $, $\alpha_3(t)=2-\frac{4}{5}\sin(\frac{2}{5}\pi t)$ and $g(u)=\frac{2}{1+u}$.
			Let $x_0=(0,0,0)^T$, $y_0=(1,1,\frac{5}{6})^T$ and $ X=[x_0,y_0]_K$. Then, Theorem 	\ref{gene-regulatory} yields that there exists a $5$-periodic solution $r(t) \in X $ such that for any $\bar{x} \in X $,
			$$\lVert \psi(t,0,\bar{x})-r(t)\lVert\ \rightarrow 0 , \, \,\,\,\, \mathrm{as} \,\,\, t \to \infty  , $$
			where $\psi(t,0,\bar{x})$ is the solution of \eqref{g(xn)-alpha1x1} satisfying $\psi(0,0,\bar{x})=\bar{x}$. \\
			Take $\bar{x}=\left(\frac{k}{4},\frac{k}{4},\frac{5k}{24}\right)^T$, for any $k=0,\cdots,4$. Our numerical simulations are illustrated as follows:
			
			\begin{figure}[H]
				\centering
				\begin{minipage}[b]{0.48\textwidth}
					\includegraphics[width=\textwidth]{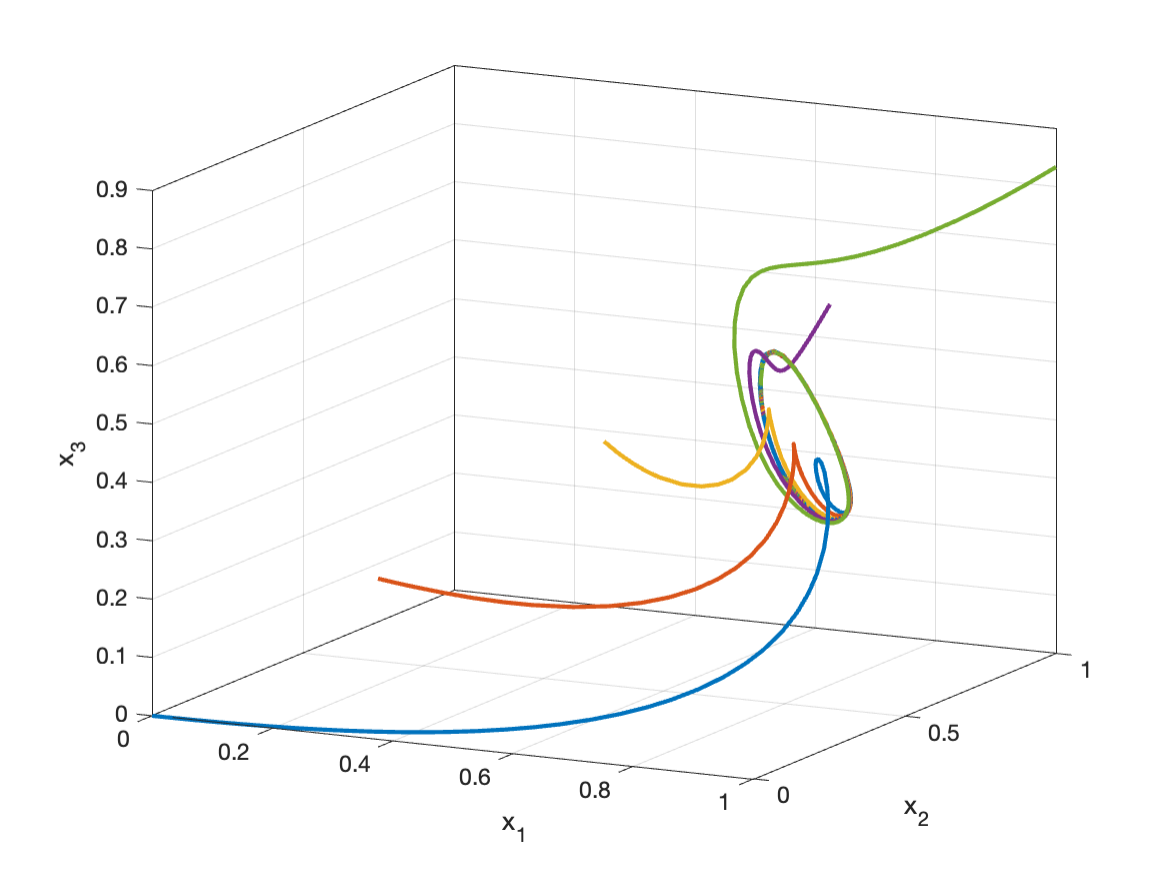}
					\caption{Trajectories of $ \psi(t,0,\bar{x})$.}
				\end{minipage}
				\hfill 
				\begin{minipage}[b]{0.48\textwidth}
					\includegraphics[width=\textwidth]{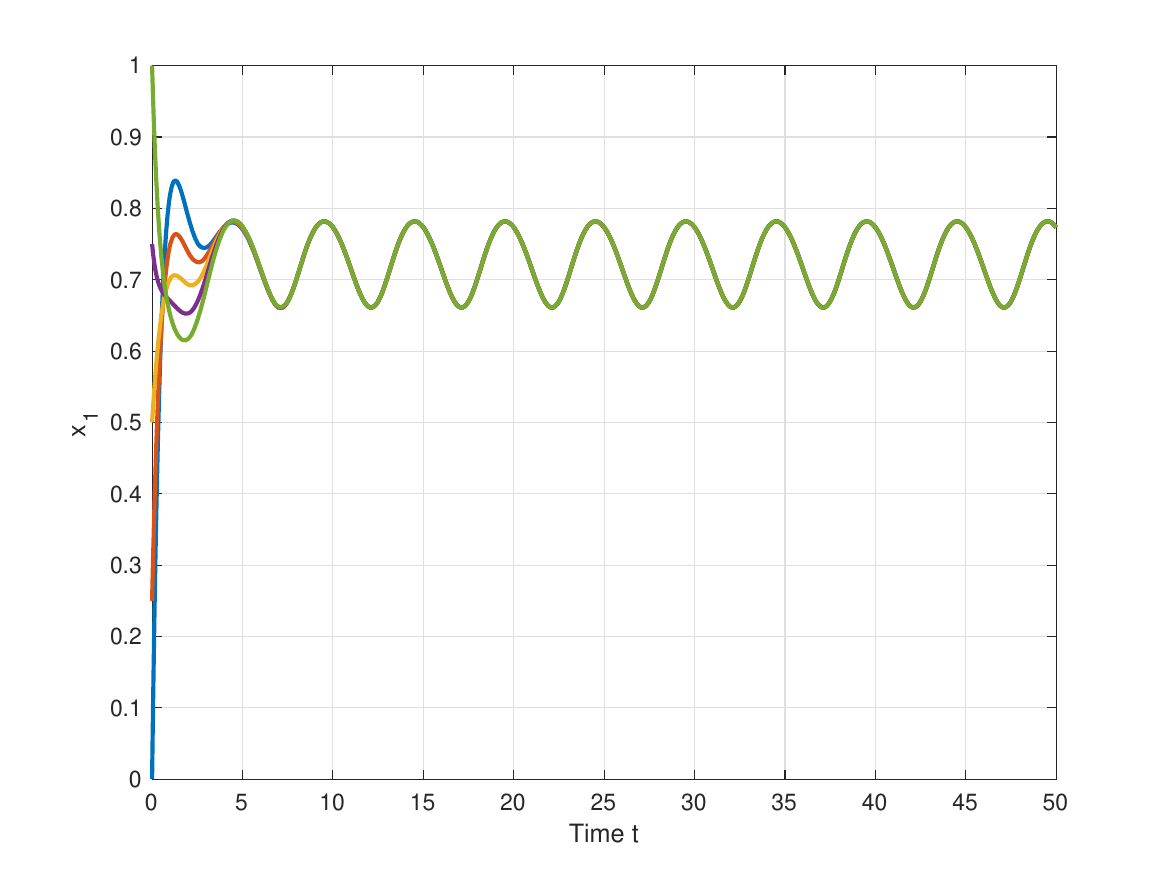}
					\caption{Trajectories of $x_1(t)$.}
				\end{minipage}
			\end{figure}
			
			\begin{figure}[htbp]
				\centering
				\begin{minipage}[b]{0.48\textwidth}
					\includegraphics[width=\textwidth]{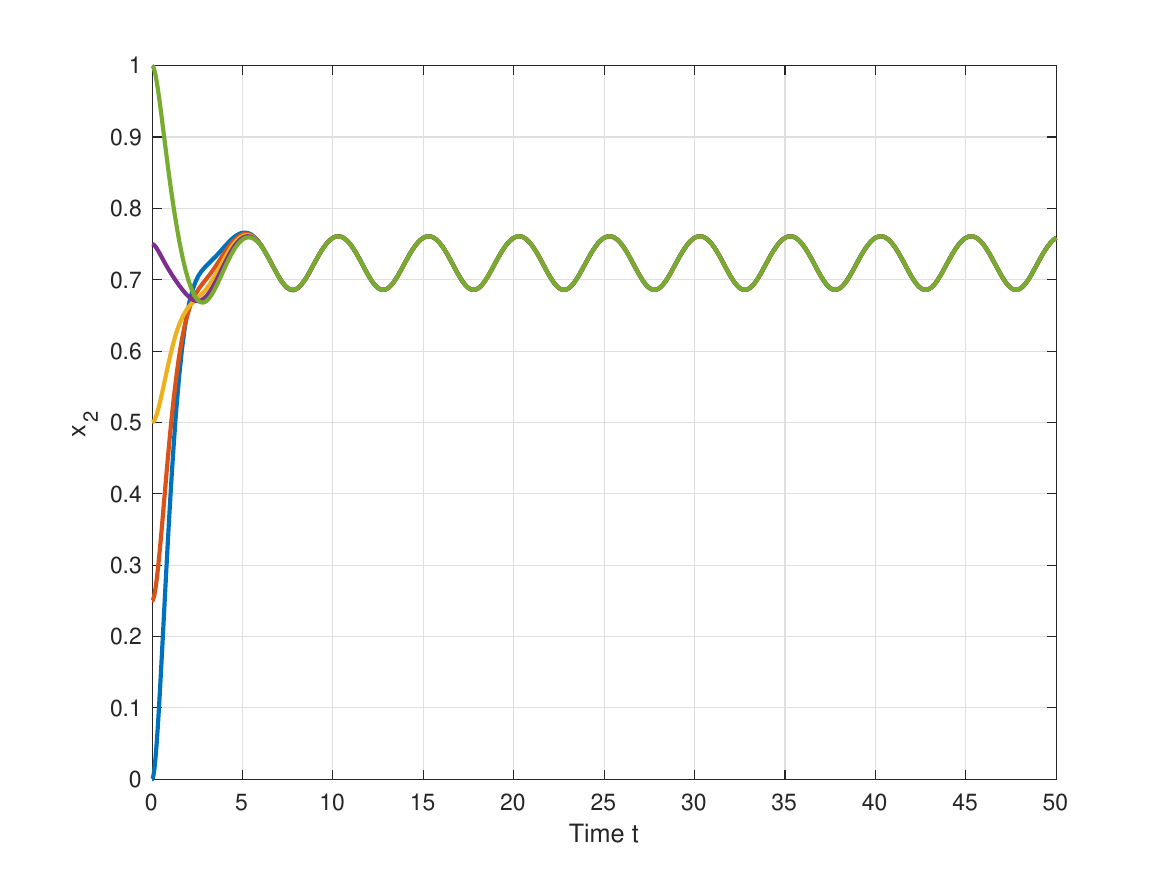}
					\caption{Trajectories of $x_2(t)$.} 
				\end{minipage}
				\hfill 
				\begin{minipage}[b]{0.48\textwidth}
					\includegraphics[width=\textwidth]{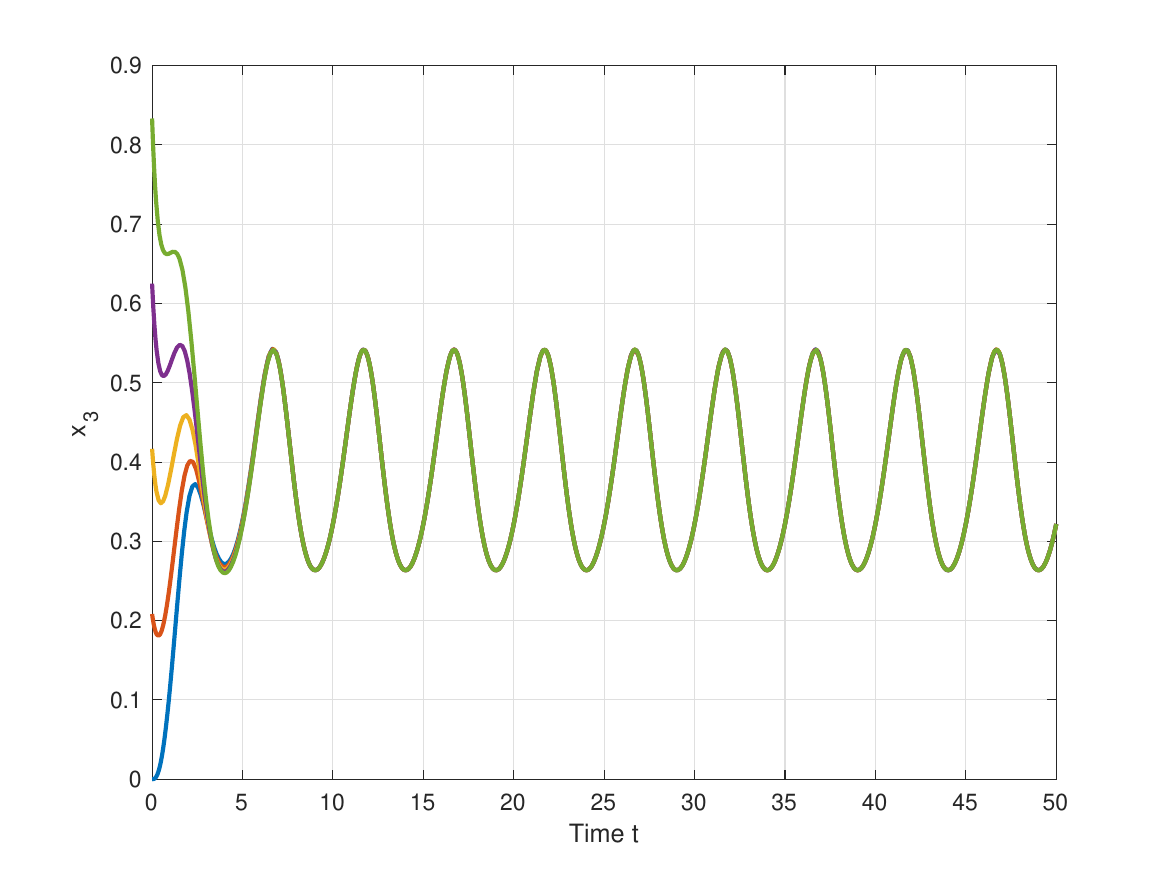}
					\caption{Trajectories of $x_3(t)$.}
				\end{minipage}
			\end{figure}

Figure 2 shows that the solution  $ \psi(t,0,\bar{x})$ to  system \eqref{g(xn)-alpha1x1}, initialized at
$\bar{x}$, ultimately converge to one common periodic orbit. Figures 3-5 respectively present the trajectory graphs for the three corresponding components of $ \psi(t,0,\bar{x})=\left(x_1(t),x_2(t),x_3(t)\right)$. These figures collectively demonstrate that each component asymptotically approaches a 5-periodic solution whose period is consistent  with the period of the time-dependent vector field.

\end{document}